\documentclass{amsart}
\usepackage{amsfonts}
\usepackage{dsfont}
\usepackage{amsmath}
\usepackage{mathrsfs}
\newtheorem{theorem}{Theorem}[section]
\newtheorem{lemma}[theorem]{Lemma}

\theoremstyle{definition}
\newtheorem{definition}[theorem]{Definition}
\newtheorem{example}[theorem]{Example}

\newtheorem{corollary}[theorem]{Corollary}
\theoremstyle{remark}

\numberwithin{equation}{section}

\begin{document}

\title[On a generalization of the Jordan canonical form theorem]{On
a generalization of the Jordan canonical form theorem on separable
Hilbert spaces}

\author{Rui Shi}


\curraddr{Department of Mathematics, Hebei Normal University,
Shijiazhuang, Hebei 050016, China} \email{littlestoneshg@gmail.com}
\thanks{The author was supported in part by NSFC Grant \#10731020.}





\begin{abstract}
We prove a generalization of the Jordan canonical form theorem for a
class of bounded linear operators on complex separable Hilbert
spaces.
\end{abstract}

\maketitle

\section{Introduction}

Throughout this paper, all Hilbert spaces discussed are
\textit{complex and separable} while all operators are bounded and
linear on the Hilbert spaces. Let $\mathscr{H}$ be a Hilbert space
and let $\mathscr{L}(\mathscr{H})$ be the set of bounded linear
operators on $\mathscr{H}$. An \textit{idempotent} $P$ is an
operator such that $P^{2}=P$. A \textit{projection} $Q$ is an
idempotent such that $Q=Q^{*}$. Unless otherwise stated, the term
\textit{algebra} will always refer to a unital subalgebra of
$\mathscr{L}(\mathscr{H})$ which is closed in the strong operator
topology. An operator $A$ on $\mathscr{H}$ is said to be
{\textit{irreducible}} if its commutant $\{A\}'\equiv
\{B\in\mathscr{L}(\mathscr{H}):AB=BA\}$ contains no projections
other than $0$ and $I$, introduced by P. Halmos \cite{Halmos}. (The
separability assumption is necessary because on a non-separable
Hilbert space every operator is reducible.) An operator $A$ on
$\mathscr{H}$ is said to be {\textit{strongly irreducible}} if
$XAX^{-1}$ is irreducible for every invertible operator $X$ in
$\mathscr{L}(\mathscr{H})$, introduced by F. Gilfeather
\cite{Gilfeather}. Equivalently, an operator $A$ is strongly
irreducible if and only if $\{A\}^{\prime}$ contains no idempotents
other than $0$ and $I$.

In \cite{Jiang Z}, strongly irreducible operators are considered as
analogues of Jordan blocks on complex separable Hilbert spaces. In
the past two decades, D. A. Herrero, Yasuo Watatani, Zejian Jiang,
C. K. Fong, Chunlan Jiang, Zongyao Wang, Peiyuan Wu, Youqing Ji,
Junsheng Fang and many other mathematicians did a lot of work around
this class of operators \cite{Jiang_1, Jiang_2, Jiang_3, Jiang_4,
Jiang_5, Jiang_6, Watatani_1, Watatani_2, Watatani_3}.

On a finite dimensional Hilbert space $\mathscr{K}$, the Jordan
canonical form theorem shows that every operator can be uniquely
written as a (Banach) direct sum of Jordan blocks up to similarity.
This means that for an operator $B$ on $\mathscr{K}$, there is a
bounded maximal abelian set of idempotents $\mathscr{Q}$ in
$\{B\}^{\prime}$ and $\mathscr{Q}$ is unique up to similarity in
$\{B\}^{\prime}$. However, to represent certain operators on
$\mathscr{H}$, direct sums of Jordan blocks need to be replaced by
direct integrals of strongly irreducible operators with regular
Borel measures. In \cite{Shi}, we proved that an operator $A$ on
$\mathscr{H}$ is similar to a direct integral of strongly
irreducible operators if and only if its commutant $\{A\}^{\prime}$
contains a bounded maximal abelian set of idempotents. Related
concepts about direct integrals can be found in
\cite{Azoff_2,Conway}.

For an operator $A$ on $\mathscr{H}$, A direct integral
decomposition of $A$ is said to be a \textit{strongly irreducible
decomposition (S. I. D.)}  of $A$ if the integrand is strongly
irreducible almost everywhere on the domain of integration. An S. I.
D. of $A$ is said to be \textit{unique up to similarity} if for
bounded maximal abelian sets of idempotents $\mathscr{P}$ and
$\mathscr{Q}$ in $\{A\}^{\prime}$, there is an invertible operator
$X$ in $\{A\}^{\prime}$ such that $X\mathscr{P}X^{-1}=\mathscr{Q}$.

In this paper, we study when $A$ has unique S. I. D. up to
similarity, for $A$ similar to a direct integral of strongly
irreducible operators.

\section{Upper triangular representation and main theorems}

If an operator $A$ in $\mathscr{L}(\mathscr{H})$ is similar to a
direct integral of strongly irreducible operators, then there is an
invertible operator $X$ in $\mathscr{L}(\mathscr{H})$ such that
$XAX^{-1}$ has an S. I. D. in the form
$$XAX^{-1}=
(\bigoplus\limits^{\infty}_{n=1}\int\limits_{\Lambda_n}
(XAX^{-1})(\lambda)d\mu(\lambda))
\bigoplus\int\limits_{\Lambda_{\infty}}
(XAX^{-1})(\lambda)d\mu(\lambda). \eqno (1)$$ Here $\mu$ is a
regular Borel measure. Write $\Lambda$ for
$\bigcup^{n=\infty}_{n=1}\Lambda_{n}$. The sets $\Lambda_{\infty}$
and $\Lambda_{n}$ for $n$ in $\mathds{N}$ are bounded Borel and
pairwise disjoint. For $n$ in $\mathds{N}$ and almost every
$\lambda$ in $\Lambda_{n}$, the dimension of the fibre space
$\mathscr{H}_{\lambda}$ is $n$. For almost every $\lambda$ in
$\Lambda_{\infty}$, the dimension of the fibre space
$\mathscr{H}_{\lambda}$ is $\infty$. (For fibre space, see
\cite{Abrahamse}, \S2.) Some $\Lambda_{n}$s and $\Lambda_{\infty}$
may be of measure zero. The partitioned measure space corresponding
to the S. I. D. of $XAX^{-1}$ is denoted by $\{\Lambda, \mu,
\{\Lambda_{n}\}^{n=\infty}_{n=1}\}$.

For a nonzero normal operator $N$ on $\mathscr{H}$, the tensor
product $I\otimes N$ does not have unique S. I. D. up to similarity,
where $I$ is the identity operator on $\mathscr{H}$ and
${\textrm{dim}}\mathscr{H}=\infty$. If $A$ is similar to a normal
operator $N$, then the S. I. D. of $A$ is unique up to similarity if
and only if the multiplicity function $m_{_N}$ for $N$ is finite
a.~e.\ on $\sigma(N)$. Based on this, we can construct a non-normal
operator which does not have unique S. I. D. up to similarity, if
$\mu(\Lambda_{\infty})\neq 0$ in (1). If $\mu(\Lambda_{\infty})=0$
then the S. I. D. of $XAX^{-1}$ is of the form
$$XAX^{-1}= \bigoplus\limits^{\infty}_{n=1}\int\limits_{\Lambda_n}
(XAX^{-1})(\lambda) d\mu(\lambda). \eqno (2)$$ By (\cite{Azoff_1},
Corollary 2), there is a unitary operator $U$ such that the equation
$$U(XAX^{-1})U^{*}(\lambda)=U(\lambda)(XAX^{-1})(\lambda)U^{*}(\lambda)$$
holds a.~e.\ on $\Lambda$ and $U(XAX^{-1})U^{*}(\lambda)$ is upper
triangular in $M_{n}(\mathds{C})$ for $\lambda$ a.~e.\ in
$\Lambda_{n}$. Write $\mu_{n}$ for $\mu|_{\Lambda_{n}}, 1\leq
n<\infty$. Without loss of generality, we assume that
$$XAX^{-1}=
\bigoplus\limits^{\infty}_{n=1}\int\limits_{\Lambda_n}
\left(\begin{array}{cccccc}
M_{\phi_{n}}&M_{\phi^{n}_{12}}&M_{\phi^{n}_{13}}&\cdots&M_{\phi^{n}_{1n}}\\
0&M_{\phi_{n}}&M_{\phi^{n}_{23}}&\cdots&M_{\phi^{n}_{2n}}\\
0&0&M_{\phi_{n}}&\cdots&M_{\phi^{n}_{3n}}\\
\vdots&\vdots&\vdots&\ddots&\vdots\\
0&0&0&\cdots&M_{\phi_{n}}\\
\end{array}\right)_{n\times n}(\lambda) d\mu(\lambda),\eqno (3)$$
where $\phi_{n},\phi^{n}_{ij}\in L^{\infty}(\mu_{n})$,
$M_{\phi_{n}}$ and $M_{\phi^{n}_{ij}}$ are multiplication operators.
The scalar-valued spectral measure for $M_{\phi_{n}}$ is
$\nu_{n}\equiv\mu_{n}\circ\phi^{-1}_{n}$. Let the set
$\{\Gamma_{nm}\}^{m=\infty}_{m=1}$ be a Borel partition of
$\sigma(M_{\phi_{n}})$ corresponding to the multiplicity function
$m_{{{\phi_{n}}}}$ for $M_{\phi_{n}}$ on $\sigma(M_{\phi_{n}})$,
where $m_{{{\phi_{n}}}}(\lambda)=m, \forall\lambda\in\Gamma_{nm}$.
Write $\nu_{nm}$ for $\nu_{n}|_{\Gamma_{nm}},1\leq m\leq\infty$.  We
find that the functions $m_{{{\phi_{n}}}}$ play a significant role
on studying the uniqueness of S. I. D. of $A$ up to similarity.
(Note that $\phi^{n}_{ij}$ does not stand for $\phi_{ij}$ to the
power of $n$ here. The symbol $n$ is only a superscript.)

\begin{definition}
The function $m_{{{\phi_{n}}}}$ on $\sigma(M_{\phi_{n}})$ is said to
be \textit{the characteristic function of S. I. D.} of $XAX^{-1}$
corresponding to $\Lambda_{n}$.
\end{definition}

In the rest of this paper, we write the partitioned measure space
corresponding to the S. I. D. of $XAX^{-1}$ as $\{\Lambda, \mu,
\{\Lambda_{n},m_{{{\phi_{n}}}}\}^{\infty}_{n=1}\}$.

The purpose of this paper is to give a sufficient condition such
that the S. I. D. of $A$ in (2) is unique up to similarity.
Precisely we prove the following theorems.

\begin{theorem}
For a fixed $n$, assume that $T\in\mathscr{L}(\mathscr{H})$ is a
direct integral of upper triangular strongly irreducible operators
and the corresponding measure space is $\{\Lambda_{n}, \mu_{n},
\{\Lambda_{n},m_{{{\phi_{n}}}}\}\}$ as in (3). If there is a unitary
operator $U$ such that both
$$UM_{\phi_{n}}U^{*}=M^{(\infty)}_{z_{\infty}}\oplus
M^{}_{z_{1}}\oplus M^{(2)}_{z_{2}}\oplus\cdots$$ and
$$UM_{\phi^{n}_{ij}}U^{*}=M^{(\infty)}_{\psi^{n,ij}_{\infty}}\oplus
M^{}_{\psi^{n,ij}_{1}}\oplus M^{(2)}_{\psi^{n,ij}_{2}}\oplus\cdots$$
hold, where $\psi^{n,ij}_{m}$ and $z_{m}$ are in
$L^{\infty}(\nu_{nm})$, and $z_{m}(t)=t, \forall t\in\Gamma_{nm}$.
then the following statements are equivalent:
\begin{enumerate}
\item The bounded $\nu_{n}$-measurable multiplicity function
$m_{{{\phi_{n}}}}$ is simple on $\sigma(M_{\phi_{n}})$.
\item The S. I. D. of $T$ is unique up to similarity.
\item The $K_{0}$ group of $\{T\}^{\prime}$ is of the form
$$K_{0}(\{T\}')\cong\{f:\sigma(M_{\phi_{n}})\rightarrow\mathds{Z}|f\
{\textrm{is\ bounded Borel.}}\}.$$
\end{enumerate}
\end{theorem}

The condition in this theorem is significant and reasonable. We show
this in the proofs in \S3. The following theorem is a generalized
version of the above theorem.

\begin{theorem}
Assume that $T\in\mathscr{L}(\mathscr{H})$ is a direct integral of
upper triangular strongly irreducible operators and the
corresponding partitioned measure space is $\{\Lambda, \mu,
\{\Lambda_{n},m_{{{\phi_{n}}}}\}^{\infty}_{n=1}\}$ as in (3). The
set $\Lambda_{n}$ is of $\mu$-measure $0$ for all but finitely many
$n$ in $\mathds{N}$. If the spectral measures for
$\{M_{\phi_{n}}\}^{\infty}_{n=1}$ are mutually singular and there is
a unitary operator $U\in\{T\}^{\prime}$ satisfying the condition in
Theorem 2.2 w.~r.~t.\ $\phi_{n}$ and $\phi^{n}_{ij}\ (i>j)$ on every
$\Lambda_{n}$, then the following statements are equivalent:
\begin{enumerate}
\item The bounded $\nu_{n}$-measurable multiplicity function
$m_{{{\phi_{n}}}}$ is simple on $\sigma(M_{\phi_{n}})$ for $n$ in
$\mathds{N}$.
\item The S. I. D. of $T$ is unique up to similarity.
\item The $K_{0}$ group of $\{T\}^{\prime}$ is of the
form
$$K_{0}(\{T\}')\cong\{f:\bigcup^{\infty}_{n=1}\sigma(M_{\phi_{n}})\rightarrow\mathds{Z}|f\
{\textrm{is\ bounded Borel.}}\}.$$
\end{enumerate}
\end{theorem}

The rest of this paper is organized as follows. In Section 3, first
we prove a special case of Theorem 3.3 in Lemma 3.2 and then we
prove Theorem 3.3 in three lemmas. Corollary 3.7 is to characterize
the $K_{0}$ group of the commutant of the operator $T$ in Theorem
3.3. In Example 3.8, we construct an operator and compute the
corresponding $K_{0}$ group. Finally, we prove Theorem 2.2 and
Theorem 2.3. The operator $T$ in Theorem 3.3 indicates why we add a
condition about the unitary operator in Theorem 2.2.

\section{Proofs}

The following lemma reveals an important property which is applied
in other lemmas in this paper.

\begin{lemma}
Assume that $\phi,\ \phi_{ij}\in L^{\infty}(\mu_{n})$, where $1\leq
i,\ j\leq n$ and $\Lambda_{n}$ is as in (1). Multiplication
operators $M_{\phi}$ and $M_{\phi_{ij}}$ are on $L^{2}(\mu_{n})$.
Then the following upper triangular form
$$T=\left(\begin{array}{cccccc}
M_{\phi}&M_{\phi_{12}}&M_{\phi_{13}}&\cdots&M_{\phi_{1n}}\\
0&M_{\phi}&M_{\phi_{23}}&\cdots&M_{\phi_{2n}}\\
0&0&M_{\phi}&\cdots&M_{\phi_{3n}}\\
\vdots&\vdots&\vdots&\ddots&\vdots\\
0&0&0&\cdots&M_{\phi}\\
\end{array}\right)_{n\times n} \eqno (4)$$ is strongly irreducible
a.~e.\ on $\Lambda_{n}$ if and only if $\phi_{i,i+1}(\lambda)\neq 0$
holds a.~e.\ on $\Lambda_{n}$ for $i=1, 2, \ldots, n-1$. The
underlying Hilbert space of $T$ is denoted by $\mathscr{H}$.
\end{lemma}

\begin{proof}
For $\lambda$ in $\Lambda_{n}$, if $T(\lambda)$ is strongly
irreducible in $M_{n}(\mathds{C})$, then there is an invertible
operator $X$ in $M_{n}(\mathds{C})$ such that
$T(\lambda)X=XJ_{\sigma(T(\lambda))}$, where
$J_{\sigma(T(\lambda))}$ is a Jordan block in $M_{n}(\mathds{C})$
with spectrum $\sigma(T(\lambda))$. The equation
$T(\lambda)X=XJ_{\sigma(T(\lambda))}$ becomes
$$\left(\begin{array}{ccccc}
\alpha&\alpha_{12}&\alpha_{13}&\cdots&\alpha_{1n}\\
0&\alpha&\alpha_{23}&\cdots&\alpha_{2n}\\
0&0&\alpha&\cdots&\alpha_{3n}\\
\vdots&\vdots&\vdots&\ddots&\vdots\\
0&0&0&\cdots&\alpha\\
\end{array}\right)
\left(\begin{array}{ccccc}
x_{11}&x_{12}&x_{13}&\cdots&x_{1n}\\
x_{21}&x_{22}&x_{23}&\cdots&x_{2n}\\
x_{31}&x_{32}&x_{33}&\cdots&x_{3n}\\
\vdots&\vdots&\vdots&\ddots&\vdots\\
x_{n1}&x_{n2}&x_{n3}&\cdots&x_{nn}\\
\end{array}\right)$$
$$=\left(\begin{array}{ccccc}
x_{11}&x_{12}&x_{13}&\cdots&x_{1n}\\
x_{21}&x_{22}&x_{23}&\cdots&x_{2n}\\
x_{31}&x_{32}&x_{33}&\cdots&x_{3n}\\
\vdots&\vdots&\vdots&\ddots&\vdots\\
x_{n1}&x_{n2}&x_{n3}&\cdots&x_{nn}\\
\end{array}\right)
\left(\begin{array}{ccccc}
\alpha&1&0&\cdots&0\\
0&\alpha&1&\cdots&0\\
0&0&\alpha&\cdots&0\\
\vdots&\vdots&\vdots&\ddots&\vdots\\
0&0&0&\cdots&\alpha\\
\end{array}\right),$$ where  $\alpha=\sigma(T{(\lambda)})$. This
equation implies that  $x_{ij}=0$ for $i>j$ and
$\alpha_{i-1,i}x_{ii}=x_{i-1,i-1}$ for $i=2, 3, \ldots, n$.
Therefore the invertibility of $X$ implies that $\alpha_{i-1,i}\neq
0$ holds for $i=2, 3, \ldots, n$.

On the other hand, if $\alpha_{i-1,i}\neq 0$ holds for $i=2, 3,
\ldots, n$, then every operator $X$ in $M_{n}(\mathds{C})$
satisfying $T(\lambda)X=XT(\lambda)$ can be expressed in the form
$$X=\left(\begin{array}{ccccc}
x_{11}&x_{12}&x_{13}&\cdots&x_{1n}\\
0&x_{11}&x_{23}&\cdots&x_{2n}\\
0&0&x_{11}&\cdots&x_{3n}\\
\vdots&\vdots&\vdots&\ddots&\vdots\\
0&0&0&\cdots&x_{11}\\
\end{array}\right).$$ If $X$ is an idempotent, then it must be $I$ or
$0$. Thus $T(\lambda)$ is strongly irreducible.

\end{proof}

\begin{lemma}
Suppose that an operator $T$ is assumed as in (4),
$\phi_{i,i+1}(\lambda)\neq 0$ holds a.~e.\ on $\Lambda_{n}$ for
$i=1, 2, \ldots, n-1$, and $\phi$ is one to one a.~e.\ on
$\Lambda_{n}$. Then the S. I. D. of $T$ is unique up to similarity.
\end{lemma}

\begin{proof}
Let $E(\cdot)$ be the spectral measure of $M_{\phi}$. Thus
$E^{(n)}(\cdot)$ is the spectral measure of $M^{(n)}_{\phi}$.
Meanwhile, the spectral measures of $M_{\phi}$ and $M_{z}$ have the
same range. First we show $\{T\}^{\prime}\subseteq
\{M^{(n)}_{\phi}\}^{\prime}$. Equivalently, we need to prove that
for every Borel subset $\sigma$ of $\sigma(M_{\phi})$ and
$X\in\{T\}^{\prime}$, the projection $E^{(n)}(\sigma)$ reduces $X$.

If we write $\mu_{n1}$ for $\mu_{n}|_{\phi^{-1}(\sigma)}$ and
$\mu_{n2}$ for $\mu_{n}|_{\Lambda_{n}\backslash\phi^{-1}(\sigma)}$,
then we have
$$[\textrm{ran}(E^{(n)}(\sigma))]=(L^{2}(\mu_{n1}))^{(n)},\
[\textrm{ran}(I-E^{(n)}(\sigma))]=(L^{2}(\mu_{n2}))^{(n)}.$$ The
operators $T$ and $X$ can be expressed in the form
$$T=\left(\begin{array}{cc}
T_{1}&0\\
0&T_{2}\\
\end{array}\right)
\begin{array}{c}
(L^{2}(\mu_{n1}))^{(n)}\\
(L^{2}(\mu_{n2}))^{(n)}\\
\end{array},\quad
X=\left(\begin{array}{cc}
X_{11}&X_{12}\\
X_{21}&X_{22}\\
\end{array}\right)
\begin{array}{c}
(L^{2}(\mu_{n1}))^{(n)}\\
(L^{2}(\mu_{n2}))^{(n)}\\
\end{array},$$ where
$$T_{i}=\left(\begin{array}{cccccc}
M^{i}_{\phi}&M^{i}_{\phi_{12}}&M^{i}_{\phi_{13}}&\cdots&M^{i}_{\phi_{1n}}\\
0&M^{i}_{\phi}&M^{i}_{\phi_{23}}&\cdots&M^{i}_{\phi_{2n}}\\
0&0&M^{i}_{\phi}&\cdots&M^{i}_{\phi_{3n}}\\
\vdots&\vdots&\vdots&\ddots&\vdots\\
0&0&0&\cdots&M^{i}_{\phi}\\
\end{array}\right)_{n\times n}
\begin{array}{c}
L^{2}(\mu_{ni})\\
L^{2}(\mu_{ni})\\
L^{2}(\mu_{ni})\\
\vdots\\
L^{2}(\mu_{ni})\\
\end{array},\quad i=1, 2.$$

The equation $T_{1}X_{12}=X_{12}T_{2}$ and the fact that
$M^{1}_{\phi}$ and $M^{2}_{\phi}$ have mutually singular
scalar-valued spectral measures imply that $X_{12}=0$. In the same
way we obtain $X_{21}=0$. Therefore
$X\in\{M^{(n)}_{\phi}\}^{\prime}$.

By Lemma 3.1, we compute the equation $TX=XT$ and obtain that the
operator $X$ has the form
$$X=\left(\begin{array}{cccccc}
M_{\psi}&M_{\psi_{12}}&M_{\psi_{13}}&\cdots&M_{\psi_{1n}}\\
0&M_{\psi}&M_{\psi_{23}}&\cdots&M_{\psi_{2n}}\\
0&0&M_{\psi}&\cdots&M_{\psi_{3n}}\\
\vdots&\vdots&\vdots&\ddots&\vdots\\
0&0&0&\cdots&M_{\psi}\\
\end{array}\right)_{n\times n},\eqno (5)$$ where $\psi,\
\psi_{ij}\in L^{\infty}(\mu_{n})$. Hence every idempotent in
$\{T\}'$ is a spectral projection of $M^{(n)}_{\phi}$. This means
that in the commutant of $T$, there is one and only one bounded
maximal abelian set of idempotents.

\end{proof}

\begin{theorem}
If an operator $T$ is assumed as in Lemma 3.2 and $m$ is a positive
integer, then the S. I. D. of $T^{(m)}$ is unique up to similarity.
\end{theorem}

We denote by $\mathscr{P}$ the set of projections in $\{T\}'$. This
is the only maximal abelian set of idempotents in $\{T\}'$. The set
$\mathscr{P}\oplus\cdots\oplus\mathscr{P}(m\ \textrm{copies})$ is a
bounded maximal abelian set of idempotents in $M_{m}(\{T\}')$. We
prove Theorem 3.3 in three lemmas.

\begin{lemma}
If $Q\in M_{m}(\{T\}')$ is an idempotent, then there is an
invertible operator $X\in M_{m}(\{T\}')$ such that the operator
$XQX^{-1}$ belongs to the set
$\mathscr{P}\oplus\cdots\oplus\mathscr{P}(m\ \textrm{copies})$.
\end{lemma}

\begin{proof}
The idempotent $Q$ is decomposable with respect to the diagonal
algebra generated by the set $\mathscr{P}^{(m)}\equiv\{P\oplus
P\oplus\cdots\oplus P (m\ \textrm{copies}):P\in\mathscr{P}\}$. The
measure space is $\Lambda_{n}$. Based on (5) in Lemma 3.2, the
operator $Q$ can be expressed in the form
$$Q=\left(\begin{array}{ccccccc}
M_{\psi^{11}}&\cdots&M_{\psi^{11}_{1n}}&&M_{\psi^{1m}}&\cdots&M_{\psi^{1m}_{1n}}\\
\vdots&\ddots&\vdots&\cdots&\vdots&\ddots&\vdots\\
0&\cdots&M_{\psi^{11}}&&0&\cdots&M_{\psi^{1m}}\\
&\vdots&&\ddots&&\vdots&\\
M_{\psi^{m1}}&\cdots&M_{\psi^{m1}_{1n}}&&M_{\psi^{mm}}&\cdots&M_{\psi^{mm}_{1n}}\\
\vdots&\ddots&\vdots&\cdots&\vdots&\ddots&\vdots\\
0&\cdots&M_{\psi^{m1}}&&0&\cdots&M_{\psi^{mm}}\\
\end{array}\right)_{mn\times mn}.$$ There is a unitary operator $U_{1}$ such
that the operator $Q_{1}=U_{1}QU^{*}_{1}$ is represented as a block
upper triangular operator-valued matrix in the form
$$Q_{1}=\left(\begin{array}{ccccc}
Q^{1}_{11}&Q^{1}_{12}&Q^{1}_{13}&\cdots&Q^{1}_{1n}\\
0&Q^{1}_{11}&Q^{1}_{23}&\cdots&Q^{1}_{2n}\\
0&0&Q^{1}_{11}&\cdots&Q^{1}_{3n}\\
\vdots&\vdots&\vdots&\ddots&\vdots\\
0&0&0&\cdots&Q^{1}_{11}\\
\end{array}\right)_{n\times n},$$ where
$$Q^{1}_{11}=\left(\begin{array}{ccc}
M_{\psi^{11}}&\cdots&M_{\psi^{1m}}\\
\vdots&\ddots&\vdots\\
M_{\psi^{m1}}&\cdots&M_{\psi^{mm}}\\
\end{array}\right)_{m\times m},\quad
Q^{1}_{ij}=\left(\begin{array}{ccc}
M_{\psi^{11}_{ij}}&\cdots&M_{\psi^{1m}_{ij}}\\
\vdots&\ddots&\vdots\\
M_{\psi^{m1}_{ij}}&\cdots&M_{\psi^{mm}_{ij}}\\
\end{array}\right)_{m\times m},$$ for $1\leq i<j\leq n$.
Notice that $Q^{1}_{11}$ is an idempotent.

Next, we prove that there is an invertible operator $X_{2}$ in
$M_{m}(L^{\infty}(\mu_{n}))$ such that $X_{2}Q^{1}_{11}X_{2}^{-1}$
is a projection in the form
$$X_{2}Q^{1}_{11}X_{2}^{-1}=\left(\begin{array}{cccc}
M_{\chi_{_{S_{1}}}}&0&\cdots&0\\
0&M_{\chi_{_{S_{2}}}}&\cdots&0\\
\vdots&\vdots&\ddots&\vdots\\
0&0&\cdots&M_{\chi_{_{S_{m}}}}\\
\end{array}\right)_{m\times m},$$ where $S_{i}$ is a Borel subset of
$\Lambda_{n}$ for $i=1, 2, \ldots, m$.

For any positive integer $k$, there is a positive integer $l_{k}$
such that given any idempotent $P$ in $M_{m}(\mathds{C})$ with norm
less than $k$ there is an invertible operator $X$ with norm less
than $l_{k}$ such that $XPX^{-1}$ is similar to the corresponding
Jordan block. That is because any idempotent in $M_{m}(\mathds{C})$
is unitarily equivalent to a block matrix in the form
$$\left(\begin{array}{cc}
I&R\\
0&0\\
\end{array}\right)$$ and
$$\left(\begin{array}{cc}
I&R\\
0&I\\
\end{array}\right)
\left(\begin{array}{cc}
I&R\\
0&0\\
\end{array}\right)
\left(\begin{array}{cc}
I&-R\\
0&I\\
\end{array}\right)
=\left(\begin{array}{cc}
I&0\\
0&0\\
\end{array}\right).\eqno(6)$$

For the set defined in (\cite{Azoff_1}, Corollary 3)
$$\begin{array}{r}
\mathscr{E}_{l_{k}}=\{(A,J,X)\in M_{m}(\mathds{C})\times
M_{m}(\mathds{C})\times
M_{m}(\mathds{C}): J\ \textrm{is\ in\ Jordan\ form},\quad\\
||X||\leq l_{k},\ ||X^{-1}||\leq l_{k}\ \ \textrm{and\ }
XAX^{-1}=J\},
\end{array}$$ the set
$\pi_{1}(\mathscr{E}_{l_{k}})$ contains every idempotent whose norm
is less than $k$. By (\cite{Azoff_1}, Theorem 1), the Borel map
$\phi_{l_{k}}:\pi_{1}(\mathscr{E}_{l_{k}})\rightarrow\pi_{3}(\mathscr{E}_{l_{k}})$
is bounded. Therefore the equivalent class of
$\phi_{l_{\lceil||Q^{1}_{11}||\rceil}}\circ Q^{1}_{11}(\cdot)$ is
the $X_{2}$ we need in $M_{m}(L^{\infty}(\mu_{n}))$.

Write $Q_{2}$ for $(X_{2}^{(n)})Q_{1}(X_{2}^{(n)})^{-1}$. There is a
unitary operator $U_{3}$ in $M_{m}(L^{\infty}(\mu_{n}))$ such that
$(U_{3}^{(n)})Q_{2}(U_{3}^{(n)})^{*}$ can be expressed in the form
$$(U_{3}^{(n)})Q_{2}(U_{3}^{(n)})^{*}=\left(\begin{array}{ccccc}
Q^{3}_{11}&Q^{3}_{12}&Q^{3}_{13}&\cdots&Q^{3}_{1n}\\
0&Q^{3}_{11}&Q^{3}_{23}&\cdots&Q^{3}_{2n}\\
0&0&Q^{3}_{11}&\cdots&Q^{3}_{3n}\\
\vdots&\vdots&\vdots&\ddots&\vdots\\
0&0&0&\cdots&Q^{3}_{11}\\
\end{array}\right)_{n\times n},$$ where
$$Q^{3}_{11}=\left(\begin{array}{cccc}
M_{\chi_{_{S_{1}}}}&0&\cdots&0\\
0&M_{\chi_{_{S_{2}}}}&\cdots&0\\
\vdots&\vdots&\ddots&\vdots\\
0&0&\cdots&M_{\chi_{_{S_{m}}}}\\
\end{array}\right)_{m\times m}.$$ The set $S_{i}$ is a Borel subset of
$\Lambda_{n}$ for $i=1, 2, \ldots, m$, and $S_{i+1}\subseteq S_{i}$
for $i=1, 2, \ldots, m-1$. Write $Q_{3}$ for
$(U_{3}^{(n)})Q_{2}(U_{3}^{(n)})^{*}$. Notice that $Q_{3}$ belongs
to $\{U_{1}T^{(m)}U^{*}_{1}\}^{\prime}$.

Next, we prove that there is an invertible operator $X_{4}$ in
$\{U_{1}T^{(m)}U^{*}_{1}\}^{\prime}$ such that
$X_{4}Q_{3}X^{-1}_{4}$ equals the following projection
$$Q_{4}=\left(\begin{array}{ccccc}
Q^{3}_{11}&0&0&\cdots&0\\
0&Q^{3}_{11}&0&\cdots&0\\
0&0&Q^{3}_{11}&\cdots&0\\
\vdots&\vdots&\vdots&\ddots&\vdots\\
0&0&0&\cdots&Q^{3}_{11}\\
\end{array}\right)_{n\times n}.$$

First, multiply each entry in the lower triangular $m\times m$
matrix form of $Q^{3}_{i,i+1}$ by $-1$ and denote this new $m\times
m$ matrix form by $X^{3}_{i,i+1}$, for $i=1, 2, \ldots, n-1$. In
$\{U_{1}T^{(m)}U^{*}_{1}\}^{\prime}$, we can construct an operator
$X^{3}_{1}$ in the form
$$X^{3}_{1}=\left(\begin{array}{ccccc}
I&X^{3}_{12}&*^{1}_{13}&\cdots&*^{1}_{1n}\\
0&I&X^{3}_{23}&\cdots&*^{1}_{2n}\\
0&0&I&\cdots&*^{1}_{3n}\\
\vdots&\vdots&\vdots&\ddots&\vdots\\
0&0&0&\cdots&I\\
\end{array}\right)_{n\times n}.$$ The operator $X^{3}_{1}$ is
invertible and $\sigma(X^{3}_{1})=\{1\}$. The fact that
$\{U_{1}T^{(m)}U^{*}_{1}\}^{\prime}$ is a subalgebra of
$\mathscr{L}(\mathscr{H}^{(m)})$ implies that the equation
$\sigma(X^{3}_{1})=\{1\}$ holds in
$\{U_{1}T^{(m)}U^{*}_{1}\}^{\prime}$. Thus $X^{3}_{1}$ is invertible
in $\{U_{1}T^{(m)}U^{*}_{1}\}^{\prime}$. Therefore the operator
$(X^{3}_{1})Q_{3}(X^{3}_{1})^{-1}$ is in the form
$$\left(\begin{array}{ccccc}
Q^{3}_{11}&0&*^{2}_{13}&\cdots&*^{2}_{1n}\\
0&Q^{3}_{11}&0&\cdots&*^{2}_{2n}\\
0&0&Q^{3}_{11}&\cdots&*^{2}_{3n}\\
\vdots&\vdots&\vdots&\ddots&\vdots\\
0&0&0&\cdots&Q^{3}_{11}\\
\end{array}\right)_{n\times n}.$$
Repeat the above procedure. We construct invertible operators
$X^{3}_{i}$ one by one in $\{U_{1}T^{(m)}U^{*}_{1}\}^{\prime}$, for
$i=1, \ldots, n-1$. After $n-1$ steps, we obtain $Q_{4}$. Denote by
$X_{4}$ the product of $X^{3}_{i}$s. The equation
$Q_{4}=X_{4}Q_{3}X^{-1}_{4}$ holds. Therefore
$X=U^{*}_{1}X_{4}U_{3}^{(n)}X_{2}^{(n)}U_{1}$ is the invertible
operator in $M_{m}(\{T\}')$ such that $XQX^{-1}$ is a projection in
$\mathscr{P}\oplus\cdots\oplus\mathscr{P}(m\ \textrm{copies})$.
\end{proof}

\begin{lemma}
If $\mathscr{Q}$ is a bounded maximal abelian set of idempotents in
$M_{m}(\{T\}^{\prime})$, then there is a subset
$\{Q_{i}\}^{2^{m}}_{i=1}\subseteq \mathscr{Q}$ such that for
$\lambda$ in $\Lambda_{n}$, the equation
$\{Q_{i}(\lambda)\}^{2^{m}}_{i=1}=\mathscr{Q}(\lambda)$ holds a.~e.\
on $\Lambda_{n}$.
\end{lemma}

\begin{proof}
By the above lemma, we know that for every $Q$ in $\mathscr{Q}$,
there is an invertible operator $X$ in $M_{m}(\{T\}^{\prime})$ such
that $XQX^{-1}$ is a projection in
$\mathscr{P}\oplus\cdots\oplus\mathscr{P}(m\ \textrm{copies})$. Thus
we define a function
$$r_{Q}(\lambda)=\frac{1}{n}\textrm{rank}(Q(\lambda)),\ \forall Q\in\mathscr{Q},\
\lambda\in\Lambda_{n}.$$ The function $r_{Q}$ is in the equivalent
class of certain simple function. To prove this lemma, we only need
to show that there are $m$ idempotents $Q_{i}$ in $\mathscr{Q}$ such
that the equation $r_{Q_{i}}(\lambda)=1$, for $i=1, \ldots, m$ holds
a.~e.\ on $\Lambda_{n}$ and $Q_{i}Q_{j}=0$, for $i\neq j$. We prove
this in two steps.

Step 1, we prove that there is an idempotent $Q'$ in $\mathscr{Q}$
such that the relation $0<r_{Q'}(\lambda)<m$ holds a.~e.\ on
$\Lambda_{n}$.

If the relation $\{r_{Q}(\lambda):Q\in\mathscr{Q}\}=\{0,m\}$ holds
a.~e.\ on $\Lambda_{n}$, then we can construct a strongly measurable
operator-valued constant function $Q^{\prime}(\cdot)$ satisfying the
following properties:
\begin{enumerate}
\item $Q^{\prime}(\cdot)$ is nontrivial a.~e.\ on
$\Lambda_{n}$.
\item The equivalent class $Q^{\prime}$ of $Q^{\prime}(\cdot)$ is a
projection in $M_{m}(\{T\}^{\prime})$ commuting with every
idempotent in $\mathscr{Q}$.
\item $Q^{\prime}$ does not belong to $\mathscr{Q}$.
\end{enumerate}
This contradicts with the assumption that $\mathscr{Q}$ is a maximal
abelian set of idempotents. Therefore, there are an idempotent
$Q^{\prime}_{1}$ in $\mathscr{Q}$ and a Borel subset $\Lambda_{n1}$
of $\Lambda_{n}$ with nonzero measure such that the relation
$0<r_{Q^{\prime}_{1}}(\lambda)<m$ holds a.~e.\ on $\Lambda_{n1}$.
Thus there are an idempotent $Q^{\prime}_{2}$ in $\mathscr{Q}$ and a
Borel subset $\Lambda_{n2}$ of $\Lambda_{n}\backslash\Lambda_{n1}$
with nonzero measure such that the relation
$0<r_{Q^{\prime}_{2}}(\lambda)<m$ holds a.~e.\ on $\Lambda_{n2}$.
Carry out this procedure and we obtain a subset
$\{Q^{\prime}_{i}\}^{\infty}_{i=1}$ of $\mathscr{Q}$ and
$\Lambda_{n}=\bigcup\limits^{\infty}_{i=1}\Lambda_{ni}$. Define
$$Q^{\prime}=\sum\limits^{\infty}_{i=1}[\chi_{_{\Lambda_{ni}}}]\cdot Q^{\prime}_{i},\quad
\chi_{_{\Lambda_{ni}}}(\lambda) =\left\{\begin{array}{cl}
I\in M_{mn}(\mathds{C}),&\lambda\in\Lambda_{ni};\\
0\in
M_{mn}(\mathds{C}),&\lambda\in\Lambda_{n}\backslash\Lambda_{ni}.
\end{array}\right.$$

The idempotent $Q^{\prime}$ is what we want in step 1.

Note that $r_{Q'}$ is in the equivalent class of a simple function.
We can write $\Lambda_{n}$ in the form of a union of disjoint Borel
subsets $\Lambda^{\prime}_{ni}$ of $\Lambda_{n}$ such that the
equation $r_{Q'}(\lambda)=i$ holds a.~e.\ on
$\Lambda^{\prime}_{ni}$. (Some $\Lambda^{\prime}_{ni}$s may be of
measure zero.) Write
$$\mathscr{Q}_{i}=\{[\chi_{_{\Lambda^{\prime}_{ni}}}]QQ^{\prime}:Q\in\mathscr{Q}\},\
\textrm{for}\ i=1, \ldots, m-1.$$

Step 2, we prove that for a fixed $i$ larger than $1$, if the set
$\Lambda^{\prime}_{ni}$ is not of measure zero, then there is an
idempotent $Q^{\prime\prime}\in\mathscr{Q}_{i}$ such that the
relation $0<Q^{\prime\prime}(\lambda)<i$ holds a.~e.\ on
$\Lambda^{\prime}_{ni}$.

Suppose that the relation
$\{r_{Q}(\lambda):Q\in\mathscr{Q}_{i}\}=\{0,i\}$ holds a.~e.\ on
$\Lambda^{\prime}_{ni}$. By the above lemma we know that there is an
invertible operator $X\in M_{m}(\{T\}')$ such that the equation
$$XQ^{\prime}X^{-1}(\lambda)=\left(\begin{array}{ccc}
\begin{array}{cc}
I&0\\
0&0\\
\end{array}&\cdots&\begin{array}{cc}
0&0\\
0&0\\
\end{array}\\
\vdots&\ddots&\vdots\\
\begin{array}{cc}
0&0\\
0&0\\
\end{array}&\cdots&\begin{array}{cc}
I&0\\
0&0\\
\end{array}\\
\end{array}\right)_{2n \times 2n}
\begin{array}{c}
\mathds{C}^{(i)}\\
\mathds{C}^{(m-i)}\\
\vdots\\
\mathds{C}^{(i)}\\
\mathds{C}^{(m-i)}\\
\end{array}$$ holds a.~e.\ on $\Lambda^{\prime}_{ni}$.
Therefore we can construct an idempotent $Q^{\prime\prime}$ in
$M_{m}(\{T\}')$ satisfying that
\begin{enumerate}
\item $Q^{\prime\prime}(\lambda)$ is a proper sub-idempotent of
$Q^{\prime}(\lambda)$ a.~e.\ on $\Lambda^{\prime}_{ni}$.
\item $Q^{\prime\prime}$ commutes with every idempotent in
$\mathscr{Q}_{i}$.
\item $Q^{\prime\prime}$ does not belong to $\mathscr{Q}_{i}$.
\end{enumerate}
Thus $\mathscr{Q}$ is not a maximal abelian set of idempotents. This
is a contradiction. Therefore there are an idempotent
$Q^{\prime\prime}_{1}\in\mathscr{Q}_{i}$ and a Borel subset
$\Lambda^{\prime}_{ni1}$ of $\Lambda^{\prime}_{ni}$ with nonzero
measure such that the relation
$0<r_{Q^{\prime\prime}_{1}}(\lambda)<i$ holds a.~e.\ on
$\Lambda^{\prime}_{ni1}$. Thus there are an idempotent
$Q^{\prime\prime}_{2}\in\mathscr{Q}_{i}$ and a Borel subset
$\Lambda^{\prime}_{ni2}$ of
$\Lambda^{\prime}_{ni}\backslash\Lambda^{\prime}_{ni1}$ with nonzero
measure such that the relation
$0<r_{Q^{\prime\prime}_{2}}(\lambda)<i$ holds a.~e.\ on
$\Lambda^{\prime}_{ni2}$. Carry out this procedure and we obtain a
subset $\{Q^{\prime\prime}_{k}\}^{\infty}_{k=1}$ of
$\mathscr{Q}_{i}$ and
$\Lambda^{\prime}_{ni}=\bigcup\limits^{\infty}_{k=1}\Lambda^{\prime}_{nik}$.
Write
$$Q^{\prime\prime}=\sum\limits^{\infty}_{k=1}[\chi_{_{\Lambda^{\prime}_{nik}}}]\cdot
Q^{\prime\prime}_{k}.$$ This idempotent $Q^{\prime\prime}$ is what
we want in step 2.

After finite steps, we obtain that there is an idempotent $Q$ in
$\mathscr{Q}$ such that the equation $r_{Q}(\lambda)=1$ holds a.~e.\
on $\Lambda_{n}$.

Repeat the above procedure, we can find $m$ idempotents $Q_{i}$ in
$\mathscr{Q}$ such that the equation $r_{Q_{i}}(\lambda)=1$, for
$i=1, \ldots, m$ holds a.~e.\ on $\Lambda_{n}$ and $Q_{i}Q_{j}=0$,
for $i\neq j$. Thus we can obtain $2^{m}$ idempotents that we need.

\end{proof}

\begin{lemma}
If $\mathscr{Q}$ is a bounded maximal abelian set of idempotents in
$M_{m}(\{T\}^{\prime})$, then there is an invertible operator  $X\in
M_{m}(\{T\}^{\prime})$ such that
$$X\mathscr{Q}X^{-1}=\mathscr{P}\oplus\cdots\oplus\mathscr{P}(m\
\textrm{copies}).$$
\end{lemma}

\begin{proof}
By the above lemma, we can find $m$ idempotents $Q_{i}$ in
$\mathscr{Q}$ such that the equation $r_{Q_{i}}(\lambda)=1$, for
$i=1, \ldots, m$ holds a.~e.\ on $\Lambda_{n}$ and $Q_{i}Q_{j}=0$,
for $i\neq j$. For $Q_{1}$, there is an invertible operator
$X_{1}\in M_{m}(\{T\}^{\prime})$ such that $X_{1}Q_{1}X^{-1}_{1}$ is
a projection in $\mathscr{P}\oplus\cdots\oplus\mathscr{P}(m\
\textrm{copies})$. The invertible operator $X_{1}$ can be chosen
such that $X_{1}Q_{1}X^{-1}_{1}$ is in the form
$$X_{1}Q_{1}X^{-1}_{1}=\left(\begin{array}{cccc}
I&0&\cdots&0\\
0&0&\cdots&0\\
\vdots&\vdots&\ddots&\vdots\\
0&0&\cdots&0\\
\end{array}\right)_{m\times m}.$$

Thus there is an invertible operator $X_{2}\in
M_{m}(\{T\}^{\prime})$ such that
$X_{2}X_{1}Q_{1}X^{-1}_{1}X^{-1}_{2}$ and
$X_{2}X_{1}Q_{2}X^{-1}_{1}X^{-1}_{2}$ are in the form
$$X_{2}X_{1}Q_{1}X^{-1}_{1}X^{-1}_{2}=\left(\begin{array}{cccc}
I&0&\cdots&0\\
0&0&\cdots&0\\
\vdots&\vdots&\ddots&\vdots\\
0&0&\cdots&0\\
\end{array}\right)_{m\times m},$$
$$X_{2}X_{1}Q_{2}X^{-1}_{1}X^{-1}_{2}=\left(\begin{array}{cccc}
0&0&\cdots&0\\
0&I&\cdots&0\\
\vdots&\vdots&\ddots&\vdots\\
0&0&\cdots&0\\
\end{array}\right)_{m\times m}.$$

By this procedure, we obtain $\{X_{i}\}^{m}_{i=1}$ such that
$XQ_{i}X^{-1}$ is in the form
$$XQ_{i}X^{-1}=\left(\begin{array}{ccccc}
0&\cdots&0&\cdots&0\\
\vdots&\ddots&\vdots&\ddots&\vdots\\
0&\cdots&I&\cdots&0\\
\vdots&\ddots&\vdots&\ddots&\vdots\\
0&\cdots&0&\cdots&0\\
\end{array}\right)_{m\times m}
\begin{array}{c}
\\
\vdots\\
i\ \textrm{th}\\
\vdots\\
\\
\end{array}
,$$ where $X=X_{m}X_{m-1}\cdots X_{1}$. This $X$ in
$M_{m}(\{T\}^{\prime})$ is the invertible operator that we need.

\end{proof}

With the above three lemmas, we finish the proof of Theorem 3.3.

\begin{corollary}
If an operator $T$ is assumed as in Theorem 3.3, then the $K_{0}$
group of $\{T\}^{\prime}$ is isomorphic to the set
$$K_{0}(\{T\}^{\prime})\cong\{f:\sigma(M_{\phi})\rightarrow\mathds{Z}|f\
\textrm{is\ bounded\ Borel.}\}.\eqno (7)$$
\end{corollary}

We give an example to show that the $K_{0}$ group of the commutant
of an operator $T$ as in the above Corollary is isomorphic to the
corresponding set as (7).

\begin{example}
Let $z(t)=t$, $t\in[0,1]$. The multiplication operator $M_z$
corresponding to $z$ is the operator on $L^2([0,1])$ defined by
$$(M_{z}f)(t)=t\cdot f(t),\quad f\in L^2([0,1]).$$ Denote by $T$ the
$2\times 2$ operator-valued matrix in the form
$$T=\left(\begin{array}{cc}
M_{z}&M_{\psi}\\
0&M_{z}\end{array}\right),\ \psi\in L^{\infty}([0,1]),\
\psi(\lambda)\neq 0\ \textrm{~a.~e.\ on~} [0,1].$$ By Lemma 3.2, we
know that every idempotent $P$ in $\{T\}^{\prime}$ is of the form
$$P=\left(\begin{array}{cc}
M_{\chi_{_{S}}}&0\\
0&M_{\chi_{_{S}}}\end{array}\right),$$ where $S$ is a Borel subset
of $[0,1]$. By Theorem 3.3, we know that, for any positive integer
$m$, in $\{T^{(m)}\}^{\prime}$ every idempotent is similar to a
projection in the form
$$P=\left(\begin{array}{cccc}
M^{(2)}_{\chi_{_{S_{1}}}}&0&\cdots&0\\
0&M^{(2)}_{\chi_{_{S_{2}}}}&\cdots&0\\
\vdots&\vdots&\ddots&\vdots\\
0&0&\cdots&M^{(2)}_{\chi_{_{S_{m}}}}
\end{array}\right), \eqno (8)$$ where $S_{i}$ is a Borel subset of
$[0,1]$. Denote the standard trace on $M_{2m}(\mathds{C})$ by Tr.
The bounded function $\frac{1}{2}\textrm{Tr}(P(\lambda))$ maps
$[0,1]$ in $\mathds{N}$ for almost every $\lambda$ in $[0,1]$.
Therefore, in $\bigcup^{\infty}_{m=1}\{T^{(m)}\}^{\prime}$, for
every idempotent $P$, denote by $[P]$ the class of idempotents
similar to $P$. In every $[P]$, there are projections as $(8)$.
Define $\rho([P])\equiv\frac{1}{2}\textrm{Tr}(P(\cdot))$. It is easy
to prove that the set
$\{[P]:P\in\bigcup^{\infty}_{m=1}\{T^{(m)}\}^{\prime}\}$ is
isomorphic to the set $\{f:[0,1]\rightarrow\mathds{N}|f\ \textrm{is\
bounded\ Borel.}\}$ and $\rho$ is the isomorphic map. By $K$-theory
for Banach algebras, we obtain
$$K_{0}(\{T\}^{\prime})\cong\{f:[0,1]\rightarrow\mathds{Z}|f\
\textrm{is\ bounded\ Borel.}\}.$$
\end{example}

\begin{lemma}
Suppose that $M_{\phi}$ is the multiplication operator on
$L^{2}(\mu_{n})$ as in Lemma 3.2. Then for every idempotent
$P\in\{M_{\phi}^{(\infty)}\}^{\prime}$, there is an invertible
operator $X$ in $\{M_{\phi}^{(\infty)}\}^{\prime}$ such that
$XPX^{-1}$ is a diagonal projection in
$\{M_{\phi}^{(\infty)}\}^{\prime}$.
\end{lemma}

\begin{proof}
The idempotent $P$ is decomposable on
$(\Gamma,\mu,\{\Gamma_{\infty}\})$, $\Gamma=\Lambda_{n}$. For every
$\lambda$ in $\Gamma$, the fiber space $\mathscr{H}_{\lambda}$ has
dimension $\infty$. First, we construct a bounded Borel set
$$\begin{array}{r}
\{(\lambda,Y,Q)\in\Lambda_{n}
\times\mathscr{L}(\mathscr{H}_{\lambda})
\times\mathscr{L}(\mathscr{H}_{\lambda}):Y^{-1}P(\lambda)Y=Q,\\ Q\
\textrm{is a diagonal projection,}\ ||Y||\leq k, ||Y^{-1}||\leq k\},
\end{array}$$ where $k$ is a positive integer large enough. By
(\cite{Azoff_2}, Proposition 2.1), we obtain an invertible operator
$X$ in $\{M_{\phi}^{(\infty)}\}^{\prime}$ such that $XPX^{-1}$ is a
diagonal projection. The operator $X$ is what we need.

\end{proof}

\begin{lemma}
If an operator $T\in\mathscr{L}(\mathscr{H})$ is assumed as in Lemma
3.2, then for every idempotent $Q$ in $\{T^{(\infty)}\}^{\prime}$,
there is an invertible operator $X$ in $\{T^{(\infty)}\}^{\prime}$
such that $XQX^{-1}$ is in
$\mathscr{P}\oplus\cdots\oplus\mathscr{P}\oplus\cdots(\infty\
\textrm{copies})$.
\end{lemma}

By the proofs of Lemma 3.4 and Lemma 3.9, we obtain this lemma.

\begin{proof}[Proof of Theorem 2.2]
By Theorem 3.3 and Corollary 3.7, we obtain $(1)\Rightarrow(2)$ and
$(1)\Rightarrow(3)$. When the multiplicity function
$m_{{{\phi_{n}}}}$ for $M_{\phi_{n}}$ takes $\infty$ on
$\sigma(M_{\phi_{n}})$, we can construct two bounded maximal abelian
sets of idempotents in the commutant of $T$ which are not similar to
each other in $\{T\}^{\prime}$. By Lemma 3.10, we know that if
$m_{{{\phi_{n}}}}$ takes $\infty$ on $\sigma(M_{\phi_{n}})$, then
the $K_{0}$ group of $\{T(\lambda)\}^{\prime}$ is $0$ a.~e.\ on
$\Gamma_{n\infty}$.
\end{proof}

\begin{proof}[Proof of Theorem 2.3]
Denote by $T_{n}$ the restriction of $T$ acting on
$(L^{2}(\mu_{n}))^{(n)}$. The operator $T$ can be expressed as
$\bigoplus\limits^{\infty}_{n=1}T_{n}$. Since the spectral measures
of $\{M_{\phi_{n}}\}^{\infty}_{n=1}$ are mutually singular, we
obtain
$$\{T\}^{\prime}=\bigoplus\limits^{\infty}_{n=1}\{T_{n}\}^{\prime}.$$
The rest of the proof is essentially an application of Theorem 2.2.
\end{proof}

\section*{Acknowledgments} The author is grateful to Professor
Chunlan Jiang and Professor Guihua Gong for their advice and
comments on writing this paper.

\bibliographystyle{amsplain}

\begin{thebibliography}{10}
\bibitem{Abrahamse} M. B. Abrahamse, \textit{Multiplication
operators.} Hilbert space operators (Proc. Conf., Calif. State
Univ., Long Beach, Calif., 1977), pp. 17--36, Lecture Notes in
Math., 693, Springer, Berlin, 1978.

\bibitem{Azoff_1} Edward A. Azoff, \textit{Borel measurability in
linear algebra.} Proc. Amer. Math. Soc. \textbf{42} (1974),
346--350.

\bibitem{Azoff_2} E. Azoff, C. Fong and F. Gilfeather, \textit{A
reduction theory for non-self-adjoint operator algebras.} Trans.
Amer. Math. Soc. \textbf{224} (1976), 351--366.

\bibitem{Conway} John B. Conway, \textit{A course in functional
analysis.} Second edition. Graduate Texts in Mathematics, 96.
Springer-Verlag, New York, 1990.

\bibitem{Davidson} Kenneth R. Davidson, \textit{$C^*$-algebras by
example.}(English summary) Fields Institute Monographs, 6. American
Mathematical Society, Providence, RI, 1996.

\bibitem{Gilfeather} F. Gilfeather, \textit{Strong reducibility of
operators.} Indiana Univ. Math. J. \textbf{22} (4) (1972), 393--397.

\bibitem{Halmos} P. Halmos, \textit{Irreducible operators.}
Michigan Math. J. \textbf{15} 1968, 215--223.

\bibitem{Jiang_1} C. Jiang and Z. Wang, \textit{The spectral picture
and the closure of the similarity orbit of strongly irreducible
operators.} Integral Equations Operator Theory \textbf{24} (1)
(1996), 81--105.

\bibitem{Jiang_2} Y. Cao, J. Fang, and C. Jiang, \textit{$K$-groups 
of Banach algebras and strongly irreducible decompositions of
operators.} J. Operator Theory \textbf{48} (2) (2002), 235--253.

\bibitem{Jiang_3} C. Jiang, \textit{Similarity classification of
Cowen-Douglas operators.} Canad. J. Math. \textbf{56} (4) (2004),
742--775.

\bibitem{Jiang_4} C. Jiang, X. Guo and K. Ji, \textit{$K$-group and
similarity classification of operators.} J. Funct. Anal.
\textbf{225} (1) (2005), 167--192.

\bibitem{Jiang_5} C. Jiang and Z. Wang, \textit{Structure of
Hilbert space operators.}  World Scientific Publishing Co. Pte.
Ltd., Hackensack, NJ, 2006.

\bibitem{Jiang_6} C. Jiang and K. Ji, \textit{Similarity
classification of holomorphic curves.} Adv. Math. \textbf{215} (2)
(2007), 446--468.

\bibitem{Shi} C. Jiang and R. Shi, \textit{Direct integrals of
strongly irrecucible operators.} J. Ramanujan Math. Soc. \textbf{26}
(2) (2011), 165--180.

\bibitem{Jiang Z} Z. Jiang and S. Sun, \textit{On completely
irreducible operators.} Front. Math. China \textbf{1} (4) (2006),
569--581.

\bibitem{Rosenthal} H. Radjavi and P. Rosenthal, \textit{Invariant
subspaces.} Ergebnisse der Mathematik und ihrer Grenzgebiete,
 Band 77. Springer-Verlag, New York-Heidelberg, 1973.

\bibitem{Schwartz} J. T. Schwartz, \textit{$W^*$-algebras.} Gordon
and Breach, New York, 1967.

\bibitem{Watatani_1}M. Enomoto and Y. Watatani, \textit{Relative
position of four subspaces in a Hilbert space.} Adv. Math.
\textbf{201} (2) (2006), 263--317.

\bibitem{Watatani_2}M. Enomoto and Y. Watatani, \textit{Exotic
indecomposable systems of four subspaces in a Hilbert space.}
Integral Equations and Operator Theory \textbf{59} (2) (2007),
149--164.

\bibitem{Watatani_3}M. Enomoto and Y. Watatani,
\textit{Indecomposable representations of quivers on
infinite-dimensional Hilbert spaces.} J. Funct. Anal. \textbf{256}
(4) (2009), 959--991.



\end{thebibliography}

\end{document}